\documentclass{amsart}

\usepackage{hyperref}
\hypersetup{
	colorlinks,%
	citecolor=blue,%
	filecolor=black,%
	linkcolor=blue,%
	urlcolor=blue
}

\newtheorem{theorem}{Theorem}[section]
\newtheorem{lemma}[theorem]{Lemma}
\newtheorem{proposition}[theorem]{Proposition}
\newtheorem{corollary}[theorem]{Corollary}

\theoremstyle{definition}
\newtheorem{definition}[theorem]{Definition}
\newtheorem{example}[theorem]{Example}

\theoremstyle{remark}
\newtheorem{remark}[theorem]{Remark}

\numberwithin{equation}{section}

\def\R{\mathbb{R}}

\def\P{\mathbb{P}}
\def\E{\mathbb{E}}
\def\SLLN{\operatorname{SLLN}}

\begin{document}

\title [On the $(p,q)$-type SLLN]{On the $(p,q)$-type Strong Law of Large Numbers for Sequences of Independent 
	Random Variables}

%\author[L. V. Th\`{a}nh]{L\^{e} V\v{a}n Th\`{a}nh}
\author{L\^{e} V\v{a}n Th\`{a}nh}
\address{Department
	of Mathematics, Vinh University, Nghe An, Vietnam}
\curraddr{}
\email{	levt@vinhuni.edu.vn}
\thanks{\textit{ Key Words and Phrases:} $(p, q)$-type strong law of large numbers, trong law of large numbers, 
	real separable Banach space, stable type $p$ Banach space, complete convergence in mean.}
\thanks{The author is grateful to Professor Andrew Rosalsky (University of Florida) for helpful comments.}

\subjclass[2010]{Primary 60F15; Secondary 60B12, 60G50}

\date{}
\begin{abstract}
	Li, Qi, and Rosalsky (Trans. Amer. Math. Soc., 2016) introduced
a refinement of the Marcinkiewicz--Zygmund strong law of large numbers (SLLN), so-called the $(p,q)$-type SLLN, 
where $0<p<2$ and $q>0$. They obtained sets of necessary and sufficient conditions for
this new type SLLN for two cases: $0<p<1$, $q>p$, and $1\le p<2,q\ge 1$.
This paper gives a complete solution to open problems
raised by Li, Qi, and Rosalsky by providing the necessary and sufficient conditions for the $(p,q)$-type SLLN for 
the cases where $0<q\le p<1$ and $0<q<1\le p<2$.
We consider random variables taking values in a real separable Banach space $\mathbf{B}$, but the results
are new even when $\mathbf{B}$ is the real line.
Furthermore, the conditions for a sequence of random variables 
$\left\{X_n, n \ge 1\right\}$ satisfying the $(p, q)$-type 
SLLN are shown to provide an exact characterization of stable type $p$ Banach spaces.
\end{abstract}

\maketitle

\section {Introduction and Main Results}\label{sec:intro}
Let $\left\{X_n, n \ge 1\right\}$ be a sequence of random variables defined on a probability space $(\Omega,\mathcal{F},\mathbb{P})$, 
and taking values in a real separable Banach space $\textbf{B}$ with norm $\|\cdot\|$. Let 
${\mathcal{F}}_n=\sigma(X_1,\cdots,X_n),\ n\ge 1$. The sequence $\left\{X_n, \mathcal{F}_n, n\geq 1\right\}$ is said to be a quasimartingale (see, e.g., Pisier \cite[p. 55]{Pisier16})
if
$\E(\|X_n\|)<\infty$ for all $ n\ge 1$, and 
\[\sum_{n=1}^\infty\E(\|\E(X_{n+1}|{\mathcal{F}}_n)-X_n\|)<\infty.\]
If the random variables are independent with mean zero, 
then it is easy to see that $\left\{(X_1+\cdots+X_n)/n^{\alpha}, \mathcal{F}_n, n\geq 1\right\}$, $\alpha>0$, 
is a quasimartingale if and only if
\[\sum_{n=1}^\infty \dfrac{\E\left(\left\|X_1+\cdots+X_n\right\|\right)}{n^{1+\alpha}}<\infty.\]

The study of limit theorems for random variables taking values in a Banach space is usually linked to the
notion of ``type'' of the space. We refer to Gin\'{e} and Zinn \cite{GineZinn83}, Hoffmann-J\o rgensen and Pisier \cite{HoffmannPisier}, Kuelbs and Zinn \cite{KuelbsZinn79}, Ledoux and Talagrand  \cite{LedouxTalagrand}, Marcus and Woyczy\'{n}ski \cite{MarcusWoyczynski}, 
Pisier \cite{Pisier} for definitions, equivalent characterizations, properties of a Banach space being of
Rademacher type $p$ or of stable type $p$, $1\le p \le 2$. 

Assume that $\{X,X_n,n\ge 1\}$ is a sequence of independent identically distributed (i.i.d.) $\mathbf{B}$-valued random variables, and $1\le p<2$. 
In order to answer question ``when is $\left\{(X_1+\cdots+X_n)/n^{1/p}, \mathcal{F}_n, n\geq 1\right\}$ a quasimartingale?'', 
Henchner \cite{Hechner}, Hechner and Heinkel \cite{HechnerHeinkel} proved the following striking result. Here and thereafter, $\ln x$ denotes
the logarithm of a positive real number $x$ to the base $\mathrm{e}=2.7182...$.

\begin{proposition}[Henchner \cite{Hechner}, Hechner and Heinkel \cite{HechnerHeinkel}]\label{theorem11}
	Let $1\le p<2$ and $\{X,X_n,n\ge 1\}$ be a sequence of i.i.d. mean zero $\mathbf{B}$-valued random variables. 
	Suppose that the Banach space $\mathbf{B}$ is of stable type $p$. Then
	\begin{equation}\label{m01}
\sum_{n=1}^\infty \dfrac{\E\left(\left\|\sum_{k=1}^{n} X_k\right\|\right)}{n^{1+1/p}}<\infty
	\end{equation}
	if and only if
	\begin{equation*}
	\begin{cases}
	\E\left(\|X\|\ln(1+\|X\|)\right)<\infty & \text{ if }\  p=1,\\[1.5mm]
	\int_{0}^{\infty}\P^{1/p}(\|X\|>t)\mathrm{d} t<\infty & \text{ if }\  1<p<2.
	\end{cases}
	\end{equation*}
\end{proposition}

Motivated by the above result, Li, Qi, and Rosalsky \cite{LQR11}, \cite{LQR15} provided conditions for
\begin{equation}\label{m04}
\sum_{n=1}^\infty \dfrac{1}{n}\E\left(\dfrac{\left\|\sum_{k=1}^{n} X_k\right\|}{n^{1/p}}\right)^q<\infty
\end{equation}
for $0<p<2$ and $q>0$. Clearly, \eqref{m04} implies that
\begin{equation}\label{th13}
\sum_{n=1}^{\infty}\dfrac{1}{n}\left(\dfrac{\|\sum_{k=1}^{n}X_k\|}{n^{1/p}}\right)^q <\infty\ \text{ almost surely (a.s.)}.
\end{equation}
Li, Qi, and Rosalsky \cite{LQR15} proved that if \eqref{th13} holds, then
\begin{equation}\label{th14}
\dfrac{\sum_{k=1}^{n}X_k}{n^{1/p}}\longrightarrow 0\ \text{ a.s.,}
\end{equation}
i.e., the sequence $\left\{X, X_n, n\ge 1\right\}$ obeys the Marcinkiewicz--Zygmund SLLN. It is well know that if $1\le p<2$ and 
$\mathbf{B}$ is of Rademacher type $p$, then \eqref{th14} holds if and only if $\mathbb{E}(\|X\|^p)<\infty$ and $\mathbb{E}(X)=0$ 
(see, e.g., de Acosta \cite{Acosta81}). For the case where $0<q<p<2$, Li, Qi, and Rosalsky \cite[Theorem 3]{LQR15} proved
that \eqref{th13} implies $\int_{0}^{\infty}\mathbb{P}^{q/p}(\|X\|^q>t)\mathrm{d} t<\infty,$ which is stronger than $\mathbb{E}(\|X\|^p)<\infty.$ 
Precisely, Li, Qi, and Rosalsky \cite{LQR15} proved the following result.
\begin{proposition}
[Li, Qi, and Rosalsky \cite{LQR15}]\label{prop12} Let $0<p<2$, $q>0$, and let $\left\{X, X_n, n\ge 1\right\}$ be a sequence of i.i.d. 
	random variables taking values in a real separable Banach space $\mathbf{B}$. 
	Then \eqref{m04} is equivalent to \eqref{th13} and
	\begin{equation}\label{m06}
	\begin{cases}
	\int_{0}^\infty\mathbb{P}^{q/p}(\|X\|^q>t)\mathrm{d} t<\infty & \text{ if }\ q<p,\\[1.5mm]
	\mathbb{E}\left(\|X\|^p\ln(1+\|X\|)\right)<\infty & \text{ if }\ q=p,\\[1.5mm]
	\mathbb{E}\left(\|X\|^q\right)<\infty & \text{ if }\ q>p.
	\end{cases}
	\end{equation}
	Furthermore, each of \eqref{m04} and \eqref{th13} implies the Marcinkiewicz-Zygmund SLLN \eqref{th14}. For $0<q<p<2$,  \eqref{m04} and \eqref{th13} are equivalent 
	so that each of them implies that \eqref{th14}
	and \eqref{m06} hold.
\end{proposition}

Motivated by the results in \cite{HechnerHeinkel, LQR11, LQR15}, Li, Qi, and Rosalsky \cite{LQR16} introduced an interesting type of SLLN as follows:

\begin{definition}[Li, Qi, and Rosalsky \cite{LQR16}]{\rm  Let $0<p<2$, $q>0$, and let $\left\{X, X_n, n\ge 1\right\}$ be a sequence of i.i.d. $\mathbf{B}$-random variables. 
		We say that $X$ satisfies the $(p, q)$-type  SLLN (and write $X\in \SLLN(p,q)$) if \eqref{th13} holds.}
\end{definition}

Li, Qi, and Rosalsky \cite{LQR16} obtained sets of necessary and sufficient conditions for $X\in \SLLN(p,q)$ for two cases: $  0<p<1, q>p$ and $1 \le p<2, q\ge 1$ (\cite[Theorems 2.1, 2.2 and 2.3]{LQR16}).
For other cases, 
necessary and sufficient conditions for $X\in\SLLN(p,q)$
remain open problems even when $\mathbf{B}=\R$ as noted by Li, Qi, and Rosalsky \cite[p. 541]{LQR16}.
In this note, we give a complete solution to these open problems by providing the necessary and sufficient conditions for the $(p, q)$-type SLLN for the
remaining cases: $0<q\le p<1$ and $0<q<1\le p<2$. 
Our main results for the real-valued random variable case can be
summarized in the following theorem. In this paper, the indicator function of a set $A$ will be denoted by $\mathbf{1}(A)$. 
\begin{theorem}\label{summary}
	Let $0<p<2$ and $q>0$. Let $\{X_n,n\ge1\}$ be a sequence of independent copies of a real-valued random variable $X$, and $u_n$ the quantile of order $1-1/n$  of $|X|$, $n\ge 1$. 
	The following two statements
	are equivalent:
	
{\rm (i)} $\ X\in \SLLN(p,q)$.

	\[\text{{\rm (ii)}}\begin{cases}
		\int_{0}^\infty\mathbb{P}^{q/p}(|X|^q>t)\mathrm{d} t<\infty & \text{ if }\ 0<q<p<1,\\[1.5mm]
		\mathbb{E}(|X|^p)<\infty \text{ and }\\[1.5mm]
		\sum_{n=1}^\infty \dfrac{\mathbb{E}\left(|X|^p \mathbf{1}(\min\{u_{n}^p,n\}<|X|^p\le n)\right)}{n}<\infty  & \text{ if }\ 0<q=p<1,\\[1.5mm]
		\mathbb{E}(X)=0 \text{ and } \int_{0}^\infty\mathbb{P}^{q/p}(|X|^q>t)\mathrm{d} t<\infty& \text{ if }\ 0<q<1\le p<2.
	\end{cases}
	\]
	The following two statements
	are equivalent:
	
	{\rm $\ $  (iii)} $\ \sum_{n=1}^\infty\dfrac{1}{n}\mathbb{E}\left(\dfrac{|\sum_{i=1}^n X_i|}{n^{1/p}}\right)^q<\infty$.
		\[\text{{\rm (iv)}}\begin{cases}
		\int_{0}^\infty\mathbb{P}^{q/p}(|X|^q>t)\mathrm{d} t<\infty & \text{ if }\ 0<q<p<1,\\[1.5mm]
		\mathbb{E}\left(|X|^p\ln(1+|X|)\right)<\infty & \text{ if }\ 0<q=p<1,\\[1.5mm]
		\mathbb{E}(X)=0\ \text{ and }\ \int_{0}^\infty\mathbb{P}^{q/p}(|X|^q>t)\mathrm{d} t<\infty& \text{ if }\ 0<q<1\le p<2.
		\end{cases}
		\]
\end{theorem} 
Versions of the above results in the Banach space setting are also given, and, especially, 
the conditions for the sequence  $\left\{X_n, n \ge 1\right\}$ satisfying the $(p, q)$-type 
SLLN are shown to provide an exact characterization of stable type $p$ Banach spaces.
The latter result was not discovered by Li, Qi, and Rosalsky \cite{LQR16} even for the case
$1\le p<2$, $q\ge 1$. 
The results are obtained by developing some techniques in Hechner and Heinkel \cite{HechnerHeinkel}, and in Li, Qi, and Rosalsky \cite{LQR11,LQR15,LQR16}, and by using some results regarding the notion of complete convergence in mean of order $p$ developed by Rosalsky, Thanh, and Volodin \cite{RTV}.

In the rest of paper, we always consider random variables that taking values in a real separable Banach space $\mathbf{B}$
if no further clarification is needed. For a random variable $X$ and for each $n\ge 1$, $u_n$ denotes the quantile of order $1-1/n$  of $\|X\|$, i.e.,
\[u_n=\inf\left\{t: \P(\|X\|\le t)>1-\dfrac{1}{n}  \right\}=\inf\left\{t: \P(\|X\|> t)<\dfrac{1}{n}  \right\}.\]

We now present Banach space versions of Theorem \ref{summary}.
Theorem \ref{thm14} provides the necessary and sufficient conditions for $X\in\SLLN(p,q)$ for the case where $0<q\le p<1$, while
Theorem \ref{thm15} deals with the case where $0<q<1\le p<2$. 

\begin{theorem}\label{thm14} Let $0<q\le p<1$ and $\left\{X, X_n, n\ge 1\right\}$ be a sequence of i.i.d. random variables. Then
	\begin{equation}\label{th16}
	X\in \SLLN(p,q)
	\end{equation}
	if and only if
	\begin{equation}\label{m03}
	\begin{cases}
	\int_{0}^\infty\mathbb{P}^{q/p}(\|X\|^q>t)\mathrm{d} t<\infty & \text{ if }\ q<p,\\[1.5mm]
	\mathbb{E}(\|X\|^p)<\infty \text{ and }\\
	 \sum_{n=1}^\infty \dfrac{\mathbb{E}\left(\|X\|^p \mathbf{1}(\min\{u_{n}^p,n\}< \|X\|^p\le n)\right)}{n}<\infty  & \text{ if }\ q=p.
	\end{cases}
	\end{equation}
\end{theorem}

\begin{remark}\label{rem01}
		We make some comments on Theorem \ref{thm14} as follows.
		
		(i) As noted by Li, Qi, and Rosalsky \cite{LQR15}, if
		$X\in \SLLN(p,q)$ for some $q>0$, then $X\in \SLLN(p,q_1)$ for all $q_1>q$. 
		By Theorem \ref{thm14} we will show that, for $0<p<1$, there exists a random variable $X$ such that 
		$X\in\SLLN(p,p)$ but $X\notin\SLLN(p,q)$ for all $0<q<p$ (see Example \ref{t01} in Section \ref{sec:main-results}).
		
		(ii) For the case where $q=p$, each of two conditions 
		$\E(\|X\|^p)<\infty$ and 
		\[\sum_{n=1}^\infty \dfrac{\mathbb{E}\left(\|X\|^p \mathbf{1}(\min\{u_{n}^p,n\}< \|X\|^p\le n)\right)}{n}<\infty\]
		does not imply each other (see Examples \ref{t03} and \ref{t05} in Section \ref{sec:main-results}).

\end{remark}

\begin{theorem}\label{thm15}
	Let $0<q<1\le p<2$ and  $\left\{X, X_n, n\ge 1\right\}$ be a sequence of i.i.d. random variables taking values in a
	separable Banach space $\mathbf{B}$. If $\mathbf{B}$ is of stable type $p$, then
	\begin{equation}\label{m05}
	X\in \SLLN(p,q)
	\end{equation}
	if and only if
	\begin{equation}\label{m07}
	\mathbb{E}(X)=0 \text{ and } \int_{0}^\infty\mathbb{P}^{q/p}(\|X\|^q>t)\mathrm{d} t<\infty.
	\end{equation}
\end{theorem}

Li, Qi, and Rosalsky \cite{LQR16} also provided necessary and sufficient conditions for 
\begin{equation}\label{m09}
\sum_{n=1}^\infty \dfrac{1}{n}\E\left(\dfrac{\left\|\sum_{k=1}^{n} X_k\right\|}{n^{1/p}}\right)^q<\infty
\end{equation}
for
the case where $  0<p<1, q>p$ and for the case where $1 \le p<2, q\ge 1$ (see \cite[Theorem 2.1 and Corollaries 2.2 and 2.3]{LQR16}).
From Theorems \ref{thm14} and \ref{thm15}, we have the following corollary.

\begin{corollary}\label{cor01}
	Let $\left\{X, X_n, n\ge 1\right\}$ be a sequence of i.i.d. random variables taking values in a
	separable Banach space $\mathbf{B}$. 
\begin{itemize}
	\item[ (i)] 	If $0<q\le p<1$, then
	\eqref{m09} is equivalent to
	\begin{equation}\label{m11}
	\begin{cases}
	\int_{0}^\infty\mathbb{P}^{q/p}(\|X\|^q>t)\mathrm{d} t<\infty & \text{ if }\ q<p,\\[1.5mm]
	\mathbb{E}\left(\|X\|^p\ln(1+\|X\|)\right)<\infty & \text{ if }\ q=p.
	\end{cases}
	\end{equation}
	
	\item[ (ii)]  If $0<q<1\le p<2$, and $\mathbf{B}$ is of stable type $p$, then \eqref{m09} is equivalent to
	\begin{equation}\label{m13}
	\mathbb{E}(X)=0\ \text{ and }\ \int_{0}^\infty\mathbb{P}^{q/p}(\|X\|^q>t)\mathrm{d} t<\infty.
	\end{equation}
	
\end{itemize}
\end{corollary}

By combining Theorems \ref{thm14} and \ref{thm15}, and Corollary \ref{cor01}, we
obtain Theorem \ref{summary}. Characterizations of SLLN in Banach spaces was proved by Hoffmann-J\o rgensen and Pisier \cite{HoffmannPisier}, de Acosta \cite{Acosta81},
and Mikosch and Norvai\v{s}a \cite{MikoschNorvaisa}.
Ledoux and Talagrand \cite{LedouxTalagrand88}, and more recently, Einmahl and Li \cite{EinmahlLi08} discovered characterizations of the law of the iterated logarithm
for Banach-valued random variables. Our Theorems \ref{thm14} and \ref{thm15}, Corollary \ref{cor01}, and the findings by Li, Qi, and Rosalsky \cite{LQR16}, and
Hechner and Heinkel \cite{HechnerHeinkel} 
complete a picture of characterizations of the $(p,q)$-type SLLN in Banach spaces, as well as characterizations of
\[\sum_{n=1}^\infty \dfrac{1}{n}\E\left(\dfrac{\left\|\sum_{k=1}^{n} X_k\right\|}{n^{1/p}}\right)^q<\infty,\ 0<p<2,\ q>0.\]

The rest of the paper is organized as follows. 
In Section \ref{sec:pqSLLN-implies-SLLN}, we prove that the $(p,q)$-type SLLN implies
the Marcinkiewicz--Zygmund SLLN without
assuming that the random variables are identically distributed.
This result allows us
to provide an exact characterization of stable
type $p$ Banach spaces through the $(p,q)$-type SLLN, which we present and prove in Section \ref{sec:stable}.
In Section \ref{sec:main-results}, we will prove Theorems \ref{thm14}, \ref{thm15}, and Corollary \ref{cor01}.
Finally, the paper is concluded
with further remarks in Section \ref{sec:remark}.

\section{The $(p,q)$-type SLLN implies the Marcinkiewicz--Zygmund SLLN}\label{sec:pqSLLN-implies-SLLN}

A sequence of random variables $\left\{X_n, n\ge 1\right\}$ is said to be stochastically dominated by a 
random variable $X$ if
\begin{equation}\label{stoch-dominated}
\sup_{n\ge 1}\mathbb{P}(\|X_n\|>t)\le \mathbb{P}(\|X\|>t),\ t\ge 0.
\end{equation}
It is well known that for a sequence of independent mean zero random variables $\left\{X_n, n \ge 1\right\}$
taking values in a real separable stable type $p$ Banach space $\mathbf{B}$, $1\le p<2$, the condition that $\left\{X_n, n \ge 1\right\}$
are stochastically dominated by a random variable $X$ with $\E(\|X\|^p)<\infty$ implies
the Marcinkiewicz--Zygmund SLLN, i.e.,
\[\lim_{n\to\infty}\dfrac{\sum_{i=1}^n X_i}{n^{1/p}}=0 \text{ a.s.}\]
However, this is no longer true if $\mathbf{B}$ is of Rademacher type $p$ only. To see this,
let $1\le p<2$, and $\ell_p$ denote the real separable Rademacher type $p$ Banach space of absolute $p^{\text{th}}$ power summable real
sequences $v=\{v_i,i\ge 1\}$ with norm 
\[\|v\|=\left(\sum_{i=1}^\infty|v_i|^p\right)^{1/p},\]
and define a sequence 
$\{V_n,n\ge 1\}$  of independent random
variables in $\ell_p$ by requiring the $\{V_n,n\ge 1\}$ to be independent with
\[\P(V_n=-v^{(n)})=\P(V_n=v^{(n)})=\dfrac{1}{2},\ n\ge 1,\]
where for $n\ge 1$, $v^{(n)}$ is the element of $\ell_p$ having $1$ in its $n^{\text{th}}$ position
and $0$ elsewhere.
Then the sequence $\{V_n,n\ge 1\}$ is stochastically dominated by $V_1$
with $\E(\|V_1\|)^p=1$. However, $\{V_n,n\ge 1\}$ does not obey
the Marcinkiewicz--Zygmund SLLN since for all $n\ge 1$,
\[\dfrac{\|\sum_{i=1}^n V_i\|}{n^{1/p}}=1.\]
In this section, we will prove that for $1\le p<2,\ q>0$, and for a sequence of independent mean zero random variables $\left\{X_n, n \ge 1\right\}$ which is
stochastically dominated by a random variable $X$ with $\E(\|X\|^p)<\infty$, the $(p, q)$-type SLLN implies the Marcinkiewicz--Zygmund
SLLN. 
Li, Qi, and Rosalsky \cite[Lemma 3]{LQR15} proved this result for i.i.d. random variables $\left\{X, X_n, n \ge 1\right\}$ by using a generalization
of Ottaviani’s inequality developed by Li and Rosalsky \cite{LiRosalsky13} and the strong stationary property of the sequence
$\left\{X_n, n \ge 1\right\}$ without assuming that $\E(\|X\|^p)<\infty$. 
In our setting, $\left\{X_n, n \ge 1\right\}$ is no longer stationary. The method we present here
is completely different from that of Li, Qi, and Rosalsky \cite[Lemma 3]{LQR15}. We involve a symmetrization argument
and some techniques regarding the notion of complete convergence in mean of order $p$ developed by Rosalsky, Thanh, and Volodin \cite{RTV}. 
The result of this section
will be used to show that the conditions for the sequence $\left\{X_n, n \ge 1\right\}$ satisfying the $(p, q)$-type SLLN
in Theorem \ref{thm15} 
are shown to provide an exact characterization of stable type $p$ Banach spaces.

Firstly, we will need the following two lemmas. The first lemma is a simple modification of Theorems 1 and 2 of Etemadi \cite{Etemadi85}.

\begin{lemma}\label{Etemadi}
	Let $\alpha>0$, and let $\left\{X_n, n \ge 1\right\}$ be
	a sequence of independent random variables. 
	Then
	\begin{equation}\label{Et31}
	\lim_{n\to\infty}\dfrac{\sum_{i=1}^n X_i}{n^{\alpha}}=0 \text{ a.s.}
	\end{equation}
	if and only if
	\begin{equation}\label{Et33} 
	\sum_{n=1}^{\infty} \dfrac{1}{n}\P\left(\left\|\sum_{i=n+1}^{2n} X_i\right\|>n^\alpha \varepsilon \right)<\infty \text{ for all }\varepsilon>0
	\end{equation}
	and
	\begin{equation}\label{Et35}
	\lim_{n\to\infty}\dfrac{\sum_{i=1}^n X_i}{n^{\alpha}}=0 \text{ in probability.}
	\end{equation}
	If we assume further that $\left\{X_n, n \ge 1\right\}$ are symmetric random variables, then \eqref{Et31} and \eqref{Et33} are equivalent.
\end{lemma}
\begin{proof}
	The proof of the first part is the same as that of Theorem 2 of Etemadi \cite{Etemadi85}. 
	The proof of the last part is the same as that of Theorem 1 of Etemadi \cite{Etemadi85}. 
\end{proof}

The next lemma shows that for independent (not necessary identically distributed) 
random variables $\{X,X_n,n\ge1\}$, \eqref{m09} implies a SLLN. When $\alpha=1$ and $1\le q\le 2$,
Lemma \ref{RTV} is Theorem 3 of Rosalsky, Thanh, and Volodin \cite{RTV}. The double sum version
of Theorem 3 of Rosalsky, Thanh, and Volodin \cite{RTV} was proved in \cite{ThanhThuy}.
\begin{lemma}\label{RTV}
	Let $\alpha>0$, $q\ge 1$, and let $\{X_n.n\ge 1\}$ be a sequence of independent mean zero 
	random variables. 
	If 
	\begin{equation}\label{RTV31}
\sum_{n=1}^\infty \dfrac{1}{n}\E\left(\dfrac{\left\|\sum_{k=1}^{n} X_k\right\|}{n^{\alpha}}\right)^q<\infty,
	\end{equation}
	then
	\begin{equation}\label{RTV33}
	\dfrac{\sum_{i=1}^{n}X_i}{n^{\alpha}}\overset{{\mathcal{L}}_q}{\longrightarrow} 0,\ \text{ and }\ \dfrac{\sum_{i=1}^{n}X_i}{n^{\alpha}}\overset{\text{ a.s. }}{\longrightarrow} 0.
	\end{equation}
\end{lemma}

\begin{proof}
	For $n\ge 1$, set
	\[S_n =\sum_{i=1}^n X_i.\]
	Then $\{\E\|S_n\|^q,n\ge1 \}$ is a non-decreasing sequence (see, e.g., \cite[Lemma 2]{RTV}). 
	Therefore,
	by applying \eqref{RTV31}, we have
	\begin{equation}\label{RTV35}
	\begin{split}
	\E\left(\left\| \dfrac{S_n}{n^\alpha}\right\|^q\right) & \le \alpha q \sum_{m=n}^\infty\dfrac{1}{m^{1+\alpha q}}\E\left(\|S_n\|^q\right)\\
	& \le \alpha q \sum_{m=n}^\infty\dfrac{1}{m^{1+\alpha q}}\E\left(\|S_m\|^q\right)\to 0 \text{ as } n\to\infty
	\end{split}
	\end{equation}
	thereby proving the first half of \eqref{RTV33}.
	Moreover, it follows from \eqref{RTV31} and Markov's inequality that for
	arbitrary $\varepsilon>0$
	\begin{equation}\label{RTV36}
	\begin{split}
	&\sum_{n=1}^{\infty} \dfrac{1}{n}\P\left(\left\|\sum_{i=n+1}^{2n} X_i\right\|>n^\alpha \varepsilon \right)\\
	& \le  \left(\dfrac{2}{\varepsilon}\right)^q \left(\sum_{n=1}^{\infty} \dfrac{1}{n} \E \left(\left\|\dfrac{S_{2n}}{n^\alpha} \right\|^q\right)+
	\sum_{n=1}^{\infty} \dfrac{1}{n} \E \left(\left\|\dfrac{S_{n}}{n^\alpha} \right\|^q\right)\right)<\infty.
	\end{split}
	\end{equation}
	The second hafl of \eqref{RTV33} then follows from the first part of Lemma \ref{Etemadi}, \eqref{RTV36}, and the first part of \eqref{RTV33}.
\end{proof}

The main result of this Section is the following proposition.

\begin{proposition}
\label{pqSLLN-implies-SLLN} 
	Let $1 \le p < 2$ and $q > 0$, and let $\left\{X_n, n \ge 1\right\}$ be
	a sequence of independent mean zero random variables which is stochastically dominated by a random variable $X$ with $\mathbb{E}(\|X\|^p) < \infty$.
	We assume further that the random variables $\left\{X_n, n \ge 1\right\}$ are symmetric when $0<q<1$. If
	\begin{equation}\label{pr31}
	\sum_{n=1}^{\infty}\dfrac{1}{n}\left(\dfrac{\|S_n\|}{n^{1/p}}\right)^q<\infty  \text{ a.s.,}
	\end{equation}
	then
	\begin{equation}\label{pr32}
	\lim_{n\to\infty}\dfrac{S_n}{n^{1/p}}=0 \text{ a.s.}
	\end{equation}
\end{proposition}

\begin{proof} Set
	\[Y_n=X_n\mathbf{1}(\|X_n\|^p\le n),\ S_n^{(1)}=\sum_{i=1}^{n}Y_i,\ n\ge 1.\]
	Since $\mathbb{E}(\|X\|^p)<\infty$ and the sequence $\{X_n,n\ge1\}$ is stochastically dominated by $X$,
	\begin{equation}\label{pr33}
	\sum_{n=1}^{\infty}\mathbb{P}(\|X_n\|^p> n)\le \sum_{n=1}^{\infty}\mathbb{P}(\|X\|^p> n)<\infty.
	\end{equation}
	By the Borel--Cantelli lemma, it follows from \eqref{pr33} that
	\begin{equation}\label{pr34}
	\mathbb{P}(\|X_n\|^p> n \text{ i.o. $(n)$})=0.
	\end{equation}
	Combining \eqref{pr31} and \eqref{pr34}, we have
	\begin{equation}\label{pr35}
	\sum_{n=1}^{\infty}\dfrac{1}{n}\left(\dfrac{\|S_n^{(1)}\|}{n^{1/p}}\right)^q<\infty  \text{ a.s.}
	\end{equation}
	To prove \eqref{pr32}, recalling \eqref{pr34}, it suffices to show that
	\begin{equation}\label{pr36}
	\lim_{n\to\infty}\dfrac{S_n^{(1)}}{n^{1/p}}=0 \text{ a.s.}
	\end{equation}
	For $n \ge 1,$ set
	\[a_n=\dfrac{1}{n^{1+q/p}},\ b_n=\sum_{k=n}^{\infty}a_k.\]
Then
	\begin{equation}\label{pr36a}
	\mathbb{E}\left(\sup_{n\ge 1}b_n\|Y_n\|^q\right)\le\mathbb{E}\left(\sup_{n\ge 1}\left(1+\dfrac{p}{q}\right)\dfrac{\|Y_n\|^q}{n^{q/p}}\right)\le 1+\dfrac{q}{p}.
	\end{equation}
	
	Firstly, we consider the case where $0<q<1$.
	Since $\{X_n,n\ge 1\}$ are symmetric random variables, $\{Y_n,n\ge 1\}$ are also symmetric. By applying inequality (11) in Theorem 7 of Li, Qi, and Rosalsky \cite{LQR15}, we conclude from \eqref{pr35} and \eqref{pr36a} that
	\begin{equation}\label{pr36b}
	\sum_{n=1}^{\infty}\dfrac{1}{n}\mathbb{E}\left(\dfrac{\|S_n^{(1)}\|}{n^{1/p}}\right)^q<\infty.
	\end{equation}
	It follows from \eqref{pr36b} and Markov's inequality that for
	arbitrary $\varepsilon>0$
	\begin{equation}\label{Etemadi31}
	\begin{split}
	&\sum_{n=1}^{\infty} \dfrac{1}{n}\P\left(\left\|\sum_{i=n+1}^{2n} Y_i\right\|>n^{1/p} \varepsilon \right)\\
	& \le  \left(\dfrac{2}{\varepsilon}\right)^q \left(\sum_{n=1}^{\infty} \dfrac{1}{n} \E \left(\left\|\dfrac{S_{2n}^{(1)}}{n^{1/p}} \right\|^q\right)+
	\sum_{n=1}^{\infty} \dfrac{1}{n} \E \left(\left\|\dfrac{S_{n}^{(1)}}{n^{1/p}} \right\|^q\right)\right)<\infty.
	\end{split}
	\end{equation}
	The conclusion \eqref{pr36} then follows from the last part of Lemma \ref{Etemadi}.

	Next, we consider the case where $q\ge 1$. Let $\left\{X^{'}, X_n^{'}, n\ge 1\right\}$ be an independent copy of $\left\{X, X_n, n \ge 1\right\}.$ For $n \ge 1,$ set
	\[V_{n}=Y_n-X_n^{'}\mathbf{1}(\|X_n^{'}\|^p\le n),\]
	and
	\[ \hat{S}_n^{(1)}=\sum_{i=1}^{n}V_i.\]
	By \eqref{pr35}, we have
	\begin{equation}\label{sym31}
	\sum_{n=1}^{\infty}\dfrac{1}{n}\left(\dfrac{\|\hat{S}_n^{(1)}\|}{n^{1/p}}\right)^q<\infty \text{ a.s.}
	\end{equation}
	Similar to the proof of \eqref{pr36b}, \eqref{sym31} leads to
	\begin{equation}\label{pr316}
	\sum_{n=1}^{\infty}\dfrac{1}{n}\mathbb{E}\left(\dfrac{\|\hat{S}_n^{(1)}\|}{n^{1/p}}\right)^q<\infty.
	\end{equation}
	By Lemma 4 of  Li, Qi, and Rosalsky \cite{LQR15}, under \eqref{sym31}, \eqref{pr316} is equivalent to
	\begin{equation}\label{pr317}
	\sum_{n=1}^{\infty}\dfrac{1}{n}\mathbb{E}\left(\dfrac{\|S_n^{(1)}\|}{n^{1/p}}\right)^q<\infty.
	\end{equation}
	This implies
	\begin{equation}\label{pr38}
\sum_{n=1}^{\infty}\dfrac{1}{n}\mathbb{E}\left(\dfrac{\|\sum_{i=1}^{n}(Y_i-\mathbb{E}(Y_i))\|}{n^{1/p}}\right)^q\le 2^{q-1}\sum_{n=1}^{\infty}\dfrac{1}{n}\mathbb{E}\left(\dfrac{\|S_n^{(1)}\|}{n^{1/p}}\right)^q<\infty .
	\end{equation}
	By Lemma \ref{RTV}, we have from \eqref{pr38} that
	\begin{equation}\label{pr39}
	\dfrac{\sum_{i=1}^{n}(Y_i-\mathbb{E}(Y_i))}{n^{1/p}}\rightarrow 0  \text{ a.s.}
	\end{equation}
	Since $\E(X_n)=0$ and $\left\{X_n, n \ge 1\right\}$ is stochastically dominated by $X$ with $\mathbb{E}(\|X\|^p) < \infty$, it is routine to prove that
	\begin{equation}\label{pr310}
	\lim_{n\to\infty}\left\|\dfrac{\sum_{i=1}^{n}\mathbb{E}(Y_i)}{n^{1/p}}\right\|\le \lim_{n\to\infty}\dfrac{\sum_{i=1}^{n}\mathbb{E}\left(\|X\|\mathbf{1}(\|X\|^p>i)\right)}{n^{1/p}}=0.
	\end{equation}
	Combining \eqref{pr39} and \eqref{pr310}, we obtain \eqref{pr36}.
\end{proof}

\section{ Characterizations of stable type $p$ Banach spaces}\label{sec:stable}
This section shows that for the sufficiency part of Theorem \ref{thm15}, we can relax the identically distributed condition
of the random variables $\left\{X_n, n \ge 1\right\}$. Furthermore, the conditions for the sequence  $\left\{X_n, n \ge 1\right\}$ satisfying the $(p, q)$-type SLLN are shown to provide an exact characterization of stable type $p$ Banach spaces.

\begin{theorem}\label{thm:stable11} Let $0<q<1\le p<2$
	and let $\mathbf{B}$ be a separable Banach space. Then the following statements are equivalent.
	\begin{itemize}
		\item[(i)] $\mathbf{B}$ is of stable type $p$.
		\item[(ii)] For every sequence $\left\{X_n, n\ge 1\right\}$ of independent mean zero $\mathbf{B}$-valued random variables which is stochastically dominated
		by a random variable $X$, the condition  
		\begin{equation}\label{thm311}
		\int_{0}^{\infty}\mathbb{P}^{q/p}(\|X\|^q>t)\mathrm{d} t<\infty
		\end{equation}
		implies
		\begin{equation}\label{thm312}
		\sum_{n=1}^{\infty}\dfrac{1}{n}\left(\dfrac{\|S_n\|}{n^{1/p}}\right)^q<\infty  \text{ a.s.}
		\end{equation}
	\end{itemize}
\end{theorem}

To prove Theorem \ref{thm:stable11}, we first present some preliminaries.
Let $\{X_k,1\le k\le n\}$ be $n$ independent real-valued random variables
and $\{X_{k}^{*},1\le k\le n\}$ the nonincreasing rearrangement of
the sequence $\{|X_k|,1\le k\le n\}$. Then the Marcus--Pisier inequality \cite{MarcusPisier} (see also Pisier \cite[Lemma 4.11]{Pisier}) asserts that for all $r\ge 1$,
\begin{equation}\label{MP}
\mathbb{P}\left(\sup_{1\le k\le n}k^{1/r}X_{k}^{*}>u\right) \le \dfrac{2\mathrm{e}}{u^{r}}\sup_{t>0}\left(t^r\sum_{k=1}^n \mathbb{P}(|X_k|>t)\right) \text{ for all }u>0.
\end{equation}

When $q=1$, the following lemma is Lemma 1 of Hechner and Heinkel \cite{HechnerHeinkel}. Li, Qi, and Rosalsky \cite{LQR16} generalized Lemma 1 of Hechner and Heinkel \cite{HechnerHeinkel} for the case where $1\le q<r<2$ (see Lemma 3.1  Li, Qi, and Rosalsky \cite{LQR16}). Lemma \ref{Rosalsky02} shows that their results also hold when
$0<q\le 1< r<2$.

\begin{lemma}\label{Rosalsky02} Let $0<q\le 1< r<2$ and let $\mathbf{B}$ be a Banach space of stable type $r$. Then
	there exists a constant $C(q,r)>0$ depending only on $q$ and $r$ such that, for every finite sequence $\{X_k,1\le k\le n\}$
	of independent $\mathbf{B}$-valued random variables with $\max_{1\le k\le n}\E(\|X_k\|^q)<\infty$,
	\begin{equation}\label{HH01}
	\mathbb{E}\left(\left\|\sum_{k=1}^n (X_k-EX_k)\right\|^q\right) \le C(q,r)\left(\sup_{t>0}t^{r/q}\sum_{k=1}^n \mathbb{P}(\|X_k\|^q>t)\right)^{q/r}.
	\end{equation}
\end{lemma}

\begin{proof}
Since $0<q\le 1< r<2$, we have
	\begin{align}\label{HH02}
	\begin{split}
	\mathbb{E}\left(\left\|\sum_{k=1}^n (X_k-EX_k)\right\|^q\right) & \le \left(\mathbb{E}\left\|\sum_{k=1}^n (X_k-EX_k)\right\|\right)^{q}\\
	& \le \left( C(r)  \left(\sup_{t>0} t^r \sum_{k=1}^n \mathbb{P}(\|X_k\|>t)\right)^{1/r}\right)^q\\
	& = (C(r))^q \left( \sup_{t>0} t^r \sum_{k=1}^n \mathbb{P}(\|X_k\|>t)\right)^{q/r}\\
	& := C(q,r)  \left( \sup_{t>0} t^{r/q} \sum_{k=1}^n \mathbb{P}(\|X_k\|^q>t)\right)^{q/r},
	\end{split}
	\end{align}
	where we have applied Liapunov's inequality in the first inequality and Lemma 1 of Hechner and Heinkel \cite{HechnerHeinkel} in the
	second inequality. This completes the proof of Lemma \ref{Rosalsky02}.
\end{proof}
The following result is a variation of Lemma \ref{Rosalsky02} for the case where $0<q<r<1$.
\begin{lemma}\label{Rosalsky01} Let $0<q<r<1$. Then
	%there exists a universal $C=C(q,r)$ depending only on $q$ and $r$ such that, 
	for every finite sequence $\{X_k,1\le k\le n\}$
	of independent random variables with $\max_{1\le k\le n}\mathbb{E}\left(\|X_k\|^q\right)<\infty$,
	\begin{equation}\label{HH03}
	\mathbb{E}\left(\left\|\sum_{k=1}^n X_k\right\|^q\right) \le C_{1}(q,r)\left(\sup_{t>0}t^{r/q}\sum_{k=1}^n \mathbb{P}(\|X_k\|^q>t)\right)^{q/r},
	\end{equation}
	where 
	\[C_1(q,r)=\left(\dfrac{1}{1-r}\right)^q\left(1+\dfrac{2q\mathrm{e}}{r-q}\right).\]
\end{lemma}

\begin{proof} Let $\{\|X_{k}\|^{*},1\le k\le n\}$ be the nonincreasing rearrangement of $\{\|X_{k}\|,1\le k\le n\}$.
	Since $0<q<r<1$,
	\begin{align}\label{HH05}
	\begin{split}
	\mathbb{E}\left(\left\|\sum_{k=1}^n X_k\right\|^q\right) & \le \mathbb{E}\left(\sum_{k=1}^n \left\|X_k\right\|\right)^{q}\\
	& = \mathbb{E}\left(\sum_{k=1}^n \left(k^{1/r}\|X_k\|^*\right)k^{-1/r}\right)^{q}\\
	& \le \mathbb{E}\left(\sup_{1\le k\le n}\left(k^{q/r}(\|X_k\|^*)^q\right)\left(\sum_{k=1}^n k^{-1/r}\right)^{q}\right)\\
	& = \mathbb{E}\left(\sup_{1\le k\le n}\left(k^{q/r}(\|X_k\|^q)^*\right)\left(\sum_{k=1}^n k^{-1/r}\right)^{q}\right)\\
	& \le \left(\dfrac{1}{1-r}\right)^q \mathbb{E}\left(\sup_{1\le k\le n}\left(k^{q/r}(\|X_k\|^q)^*\right)\right)\\
	& = \left(\dfrac{1}{1-r}\right)^q \int_{0}^\infty\mathbb{P}\left(\sup_{1\le k\le n}\left(k^{q/r}(\|X_k\|^q)^*\right)>u\right)\mathrm{d} u.
	\end{split}
	\end{align}
	Let $\Delta=\sup_{t>0}t^{r/q}\sum_{k=1}^n \mathbb{P}\left(\|X_k\|^q>t\right)$. Applying \eqref{MP}, we have
	\begin{align}\label{HH07}
	\begin{split}
	&\int_{0}^\infty\mathbb{P}\left(\sup_{1\le k\le n}\left(k^{q/r}(\|X_k\|^q)^*\right)>u\right)\mathrm{d} u\\ 
	& = \left(\int_{0}^{\Delta^{q/r}}+\int_{\Delta^{q/r}}^{\infty}\right)\mathbb{P}\left(\sup_{1\le k\le n}\left(k^{q/r}(\|X_k\|^q)^*\right)>u\right)\mathrm{d} u\\
	&\le \Delta^{q/r}+ 2\mathrm{e}\int_{\Delta^{q/r}}^\infty\dfrac{\Delta}{u^{r/q}} \mathrm{d} u\\
	& = \left(1+\dfrac{2q\mathrm{e}}{r-q}\right) \Delta^{q/r}.
	\end{split}
	\end{align}
	Combining \eqref{HH05} and \eqref{HH07}, we obtain \eqref{HH03}.
\end{proof}

Motivated by Lemma $3.4$ of Li, Qi, and Rosalsky \cite{LQR16} which considered the case where $1 \le q \le p < 2$
and i.i.d. random variables, we have the following lemma.

\begin{lemma}\label{LQR01}
	Let $0<q\le p<2$, and let $\{X_n\}$ be a sequence of independent $\mathbf{B}$-valued random variables. Suppose that $\{X_n, n\ge 1\}$ is stochastically dominated by a random variable $X$ satisfying
	\begin{equation}\label{th27}
	\int_{0}^\infty \mathbb{P}^{q/p}\left(\|X\|^q>t\right)\mathrm{d} t<\infty
	\end{equation}
	For each $n\ge 1$, let the quantile $u_n$ of order $1-1/n$  of $\|X\|$ be defined as in Section \ref{sec:intro}, and set
	\[Y_{n, k}=X_{k}\mathbf{1}(\|X_k\|^p\leq u_n),\ Z_{n, k}=X_{k}\mathbf{1}(\|X_k\|^p\leq n),\] 
	\[U_{n,k}=\sum_{i=1}^{k}Z_{n,i},\ U_{n,k}^{(1)}=\sum_{i=1}^{k}Y_{n,i},\  U_{n,k}^{(2)}= U_{n,k}- U_{n,k}^{(1)}.\]
	Then the following statements holds.
	\begin{itemize}
		\item[(i)] If $0<p<1$, then

	\begin{equation}\label{th28}
	\sum_{n=1}^{\infty}\dfrac{\mathbb{E}\left(\|U_{n,n}^{(1)}\|^q\right)}{n^{1+q/p}}<\infty.
	\end{equation}
\item[(ii)] If $1\le p<2$ and $\mathbf{B}$ is of stable type $p$, then 
	\begin{equation}\label{th28b}
\sum_{n=1}^{\infty}\dfrac{\mathbb{E}\left(\|U_{n,n}^{(1)}-\mathbb{E}U_{n, n}^{(1)}\|^q\right)}{n^{1+q/p}} <\infty.
\end{equation}
	\end{itemize}

	In particular, if $\mathbb{E}(\|X\|^p)<\infty$, then the following statements holds.
	\begin{itemize}
		\item[(iii)] If $0<p<1$, then 
		\begin{equation}\label{th29}
		\sum_{n=1}^{\infty}\dfrac{\mathbb{E}\left(\|U_{n,n}^{(1)}\|^p\right)}{n^{2}}<\infty.
		\end{equation}
		\item [(iv)] If $1\le p<2$ and $\mathbf{B}$ is of stable type $p$, then 
	\begin{equation}\label{th29b}
\sum_{n=1}^{\infty}\dfrac{\mathbb{E}\left(\|U_{n,n}^{(1)}-\mathbb{E}U_{n, n}^{(1)}\|^p\right)}{n^{2}} <\infty.
\end{equation}
	\end{itemize}
\end{lemma}
\begin{proof}
	Firstly, we consider the case where $0 < q \le p < 1$. Let $p < r < 1$, and $C_1(q,r)$ be as in Lemma \ref{Rosalsky01}. By applying Lemma \ref{Rosalsky01}, we
	obtain
	\begin{align}\label{R07}
	\begin{split}
	\mathbb{E}\left(\left\|U_{n,n}^{(1)}\right\|^q\right) &\le C_1(q,r)\left(\sup_{t>0}t^{r/q}\sum_{k=1}^n \mathbb{P}(\|X_k\|^qI \{\|X_k\| \le u_n\}>t)\right)^{q/r}\\
	&= C_1(q,r)\left(\sup_{0\le t\le u_{n}^q}t^{r/q}\sum_{k=1}^n \mathbb{P}(\|X_k\|^qI \{\|X_k\| \le u_n\}>t)\right)^{q/r}\\
	& \le C_1(q,r) \left(n \sup_{0\le t \le u_{n}^q} t^{r/q} \mathbb{P}(\|X\|^q>t)\right)^{q/r}\\
	& = C_1(q,r) \left(n \sup_{0\le t \le u_{n}^q} \left(\int_{0}^t \mathbb{P}^{q/r}(\|X\|^q>t)\mathrm{d} x\right)^{r/q}\right)^{q/r}\\
	& \le C_1(q,r) \left(n \sup_{0\le t \le u_{n}^q} \left(\int_{0}^t \mathbb{P}^{q/r}(\|X\|^q>x)\mathrm{d} x\right)^{r/q}\right)^{q/r}\\
	& = C_1(q,r) n^{q/r} \int_{0}^{u_{n}^q} \mathbb{P}^{q/r}(\|X\|^q>x)\mathrm{d} x\\
	& = C_1(q,r) n^{q/r} \sum_{k=1}^n \int_{u_{k-1}^q}^{u_{k}^q} \mathbb{P}^{q/r}(\|X\|^q>x)\mathrm{d} x.
	\end{split}
	\end{align}
For $k\ge 1$ and $u_{k-1}^q \le x <u_{k}^q$, we have $\mathbb{P}(\|X\|^q>x)\ge 1/k$. It thus follows from \eqref{R07} that
	\begin{align*}%\label{R09}
	\sum_{n=1}^\infty \dfrac{\mathbb{E}\left(\|U_{n,n}^{(1)}\|^q\right)}{n^{1+q/p}} 
	&\le C_1(q,r) \sum_{n=1}^\infty \dfrac{1}{n^{1+q/p-q/r}}\sum_{k=1}^n \int_{u_{k-1}^q}^{u_{k}^q} \mathbb{P}^{q/r}(\|X\|^q>x)\mathrm{d} x\\
	& = C_1(q,r) \sum_{k=1}^\infty \left(\int_{u_{k-1}^q}^{u_{k}^q} \mathbb{P}^{q/r}(\|X\|^q>x)\mathrm{d} x\right) \left(\sum_{n=k}^\infty \dfrac{1}{n^{1+q/p-q/r}}\right)\\
	& \le \left(1+\dfrac{pr}{q(r-p)}\right) C_1(q,r) \sum_{k=1}^\infty \dfrac{1}{k^{q/p-q/r}}\int_{u_{k-1}^q}^{u_{k}^q} \mathbb{P}^{q/r}(\|X\|^q>x)\mathrm{d} x \\
	& \le \left(1+\dfrac{pr}{q(r-p)}\right) C_1(q,r) \sum_{k=1}^\infty \int_{u_{k-1}^q}^{u_{k}^q} \mathbb{P}^{q/p}(\|X\|^q>x)\mathrm{d} x \\
	& = \left(1+\dfrac{pr}{q(r-p)}\right) C_1(q,r) \int_{0}^{\infty} \mathbb{P}^{q/p}(\|X\|^q>x)\mathrm{d} x<\infty
	\end{align*}
	thereby proving \eqref{th28} for the case where $0<q\le p<1$.
	
	For the case where $1\le q\le p<2$, Li, Qi, and Rosalsky \cite{LQR16} proved \eqref{th28}
	under stronger assumption that $\left\{X, X_n, n \ge 1\right\}$ are identically distributed random variables (see \cite[Lemma 3.4]{LQR16}). When the sequence $\left\{X_n, n \ge 1\right\}$ is stochastically dominated by $X$, their proof will be unchanged except for some simple modifications and therefore we conclude that Lemma \ref{LQR01} holds for the case where $1\le q\le p<2$. 
	
	Next, we consider the case where $0 < q < 1 \le p < 2$. Li, Qi, and Rosalsky \cite{LQR16} proved their Lemma 3.4 (\cite[p. 548]{LQR16}) by applying \eqref{HH01} for $1 \le q < r < 2$. In our Lemma \ref{Rosalsky02}, we have
	showed that \eqref{HH01} holds for the case where $0 < q < 1 < r < 2$. Then by using the same argument as in the proof of Lemma 3.4 of Li, Qi, and Rosalsky \cite{LQR16}, we obtain \eqref{th28} for the case where $0 < q < 1 \le p < 2.$
	
	Finally, by taking $q = p$, \eqref{th27} holds if and only if $\mathbb{E}(\|X\|^p) < \infty$, and \eqref{th28} coincides with \eqref{th29}, \eqref{th28b} coincides with \eqref{th29b}. Therefore, the last part of the lemma follows from the first part. This completes the proof.
\end{proof}

\begin{proof}[Proof of Theorem \ref{thm:stable11}]  Firstly, we verify the implication (i)$\Rightarrow $(ii).
	For each $n\ge 1$, let the quantile $u_n$ of order $1-1/n$  of $\|X\|$ be defined as in Section \ref{sec:intro},  and set for $1\le k\le n$,
	\[Y_{n, k}=X_{k}\mathbf{1}(\|X_k\|^p\leq u_n),\ U_{n,k}^{(1)}=\sum_{i=1}^{k}Y_{n,i}.\]
	By following the proof of Lemma 3.3 of Li, Qi, and Rosalsky \cite{LQR16} and noting that every real separable Banach space is of Rademacher type $q$ for $0 < q \le 1$, we have
	\begin{equation}\label{thm313}
	\sum_{n=1}^{\infty}\dfrac{\mathbb{E}\left(\left\|\left(S_n-U_{n,n}^{(1)}\right)-\mathbb{E}\left(S_n-U_{n,n}^{(1)}\right)\right\|^q\right)}{n^{1+q/p}}<\infty.
	\end{equation}
	By \eqref{th28b}, we have
	\begin{equation}\label{thm314}
	\sum_{n=1}^{\infty}\dfrac{\mathbb{E}\left(\left\|U_{n,n}^{(1)}-\mathbb{E}U_{n,n}^{(1)}\right\|^q\right)}{n^{1+q/p}}<\infty.
	\end{equation}
	Combining \eqref{thm313} and \eqref{thm314} and noting that $\mathbb{E}(S_n) = 0$, we obtain
	\[\sum_{n=1}^{\infty}\dfrac{\mathbb{E}(\|S_n\|^q)}{n^{1+q/p}}<\infty\]
	which yields \eqref{thm312}.
	
	We will now prove the implication (ii)$\Rightarrow $(i). Let $\left\{\varepsilon_k, k\ge 1\right\}$ be a Rademacher sequence and 
	let $\left\{x_k, k\ge 1\right\}$ be a sequence of elements in $\mathbf{B}$
	such that
	\begin{equation}\label{thm315}
	X:=\sup_{k\ge 1}\|x_k\|<\infty.
	\end{equation}
	By Theorem V.9.3 in \cite{Woyczynski}, (i) holds if we show that
	\begin{equation}\label{thm316}
	\lim_{n\to\infty}\dfrac{1}{n^{1/p}}\sum_{k=1}^{n}x_k\varepsilon_k=0 \text{ a.s.}
	\end{equation}
	Set
	\[X_k=x_k\varepsilon_k, k\ge 1.\]
	Then $\left\{X_k,  k \ge 1\right\}$ is a sequence of independent symmetric $\mathbf{B}$-valued random variables, stochastically dominated by $X$.
	Since $X$ is bounded, \eqref{thm311} holds. Therefore, by (ii), we have
	\begin{equation}\label{thm317}
	\sum_{n=1}^{\infty}\dfrac{1}{n}\left(\dfrac{\|\sum_{k=1}^{n}X_k\|}{n^{1/p}}\right)^q<\infty\ \text{ a.s.}
	\end{equation}
	By applying Proposition \ref{pqSLLN-implies-SLLN}, we obtain
	\eqref{thm316}.
\end{proof}

Now, we consider the case where $q\ge 1$ and $1\le p<2$. Li, Qi, and Rosalsky \cite{LQR16} provided set of necessary and sufficient conditions for the $(p,q)$-SLLN. 
Theorem 2.2 of Li, Qi, and Rosalsky \cite{LQR16} is as follows.
\begin{proposition}
[Theorem 2.2 of \cite{LQR16}]\label{LQR16Thm22} Let $1 < p < 2$, $q \ge 1$, and let $\left\{X, X_n, n \ge 1\right\}$ be a sequence of i.i.d. random variables taking values in a
	real separable stable type $p$ Banach space $\mathbf{B}$. Then $X\in \SLLN(p,q)$ if and only if 
	$\mathbb{E}(X)=0$ and
	\begin{equation}\label{stable21}
	\begin{cases}
	\int_{0}^{\infty}\mathbb{P}^{q/p}(\|X\|^q>t)\mathrm{d} t<\infty&\text{ if }\  q<p,\\[1.5mm]
	\mathbb{E}(\|X\|^p)<\infty, \sum_{n=1}^{\infty}\dfrac{\int_{\min\left\{u_n^p, n\right\}}^{n}\mathbb{P}(\|X\|^p>t)\mathrm{d} t}{n}<\infty&\text{ if }\ q=p,\\[1.5mm]
	\mathbb{E}(\|X\|^p)<\infty&\text{ if }\   q>p.
	\end{cases}
	\end{equation}
\end{proposition}

Similar to Theorem \ref{thm:stable11}, the following theorem is a complement of Proposition \ref{LQR16Thm22} (i.e., Theorem 2.2 of Li, Qi, and Rosalsky \cite{LQR16}).
\begin{theorem}\label{thm:stable23} Let $1 < p < 2$, $q\ge 1$, and let $\mathbf{B}$ be a real separable Banach space. Then the following statements are	equivalent.
	\begin{itemize}
		\item[(i)] $\mathbf{B}$ is of stable type $p$.
		\item[(ii)] For every sequence $\left\{X_n, n\ge 1\right\}$ of independent mean zero $\mathbf{B}$-valued random variables which is stochastically dominated
		by a random variable $X$, condition
		\eqref{stable21} implies
		\[\sum_{n=1}^{\infty}\dfrac{1}{n}\left(\dfrac{\|S_n\|}{n^{1/p}}\right)^q<\infty \text{ a.s.}\]	
	\end{itemize}
\end{theorem}
\begin{proof} 
		The proof of the implication (ii)$\Rightarrow $(i) is exactly the same as that of Theorem \ref{thm:stable11}. 
		The proof of the implication (i)$\Rightarrow $(ii)
is similar to that of the sufficient part of Theorem 2.2 of Li,
	Qi and Rosalsky \cite{LQR16} with some simple changes. 
	We omit the details.
\end{proof}

Similarly, we have the following theorem for the case where $p=1$. It is a complement of Theorem 2.3 of Li, Qi, and Rosalsky \cite{LQR16}.

\begin{theorem}\label{thm:stable24} Let $q\ge 1$, and let $\mathbf{B}$ be a real separable Banach space. Then the following statements are	equivalent.
	\begin{itemize}
		\item[(i)] $\mathbf{B}$ is of stable type $1$.
		\item[(ii)] For every sequence $\left\{X_n, n\ge 1\right\}$ of independent mean zero $\mathbf{B}$-valued random variables which is stochastically dominated
		by a random variable $X$, the conditions 
	\[\mathbb{E}(\|X\|)<\infty,\ \sum_{n=1}^\infty \dfrac{1}{n^2}\left(\sum_{i=1}^n \left\|\E(X_i\mathbf{1}(\|X_i\|\le n))\right\|^q\right)<\infty,\]
	 and
\begin{equation*}
 \sum_{n=1}^{\infty}\dfrac{\mathbf{1}(\{q=1\})\int_{\min\left\{u_n, n\right\}}^{n}\mathbb{P}(\|X\|>t)\mathrm{d} t}{n}<\infty
 \end{equation*}
 implies
\begin{equation}\label{stable22}
\sum_{n=1}^{\infty}\dfrac{1}{n}\left(\dfrac{\|S_n\|}{n}\right)^q<\infty\ \text{ a.s.}
\end{equation}
		\item[(iii)] For every sequence $\left\{X_n, n\ge 1\right\}$ of independent symmetric $\mathbf{B}$-valued random variables which is stochastically dominated
		by a random variable $X$, the conditions $\mathbb{E}(\|X\|)<\infty$ and
		\begin{equation*}
		\sum_{n=1}^{\infty}\dfrac{\mathbf{1}(\{q=1\})\int_{\min\left\{u_n, n\right\}}^{n}\mathbb{P}(\|X\|>t)\mathrm{d} t}{n}<\infty
		\end{equation*}
		implies \eqref{stable22}.
	\end{itemize}
\end{theorem}

\section {Proof of the main results}\label{sec:main-results}

In this section, we will prove Theorems \ref{thm14} and \ref{thm15}.
The following lemma is proved by Li, Qi, and Rosalsky \cite[Lemma 5.4]{LQR11} for the case $K=1$. 
Its proof is similar to that of Lemma 5.4 of Li, Qi, and Rosalsky \cite{LQR11}.
\begin{lemma}\label{lem24} Let $Y_1, ..., Y_n$ be i.i.d. nonnegative real-valued random variables such that
	\begin{equation}\label{r01}
	\mathbb{P}(Y_1>0)\le \dfrac{K}{n}\  \text{ for some constant }\ K\ge 1.
	\end{equation}
	Then
	\begin{equation}\label{r02}
	\mathbb{E}\left(\max_{1\le k\le n} Y_k\right)\ge \dfrac{n}{2K}\mathbb{E}(Y_1).
	\end{equation}
\end{lemma}
\begin{proof}
	For all $t\ge 0$, we have
	\begin{equation}\label{r03}
	\begin{split}
	\P\left(\max_{1\le k\le n}Y_k>t\right)=1-\left(1-\P\left(Y_1\le t\right)\right)^n\ge 1-{\mathrm{e}}^{-n\P(Y_1>t)}.
	\end{split}
	\end{equation}
	Elementary calculus shows that
	\[1-\mathrm{e}^{-x}\ge \dfrac{x}{2K}\ \text{ for all }\ 0\le x\le K. \]
	It thus follows from \eqref{r01} and \eqref{r03} that
	\[\P\left(\max_{1\le k\le n}Y_k>t\right)\ge \dfrac{n}{2K}\P(Y_1>t)\ \text{ for all }\ t\ge 0, \]
	which implies \eqref{r02}.
\end{proof}

The following lemma is a special case of Lemma 3.2 of Li and Rosalsky \cite{LiRosalsky13}. This
useful result will be used in our symmetrization procedure.

\begin{lemma}\label{lem25} Let $g: \mathbf{B}\rightarrow [0,\infty]$ be a measurable even function such that for all $x, y \in\mathbf{B},$
	\[g(x+y)\le\beta (g(x)+g(y)),\]
	where $\beta\ge 1$ is a constant, depending only on the function $g$. If $V$ is a $\mathbf{B}$-valued random variable and $\hat{V}$ is a
	symmetrized version of $V$ , then for all $t\ge 0$, we have that
	\[\mathbb{P}(g(V)\le t)\mathbb{E}(g(V))\le\mathbb{E}(g(\hat{V}))+\beta t.\]
\end{lemma}

\begin{proof}[Proof of Theorem \ref{thm14} (Sufficiency)]
	
	Firstly we consider the case where $0<q<p<1$.
	We will prove that
	\begin{align}\label{pr01}
	\sum_{n=1}^\infty \dfrac{1}{n}\mathbb{E}\left(\dfrac{\|S_n\|}{n^{1/p}}\right)^q<\infty.
	\end{align}
	For each $n\ge 1$, let the quantile $u_n$ of order $1-1/n$  of $\|X\|$ be defined as in Section \ref{sec:intro}. For $n\ge 1$, $1\le k\le n$, set
	\[Y_{n, k}=X_{k}\mathbf{1}(\|X_k\|^p\leq u_n),\ Z_{n, k}=X_{k}\mathbf{1}(\|X_k\|^p\leq n),\] 
	and
	\[U_{n,k}=\sum_{i=1}^{k}Z_{n,i},\ U_{n,k}^{(1)}=\sum_{i=1}^{k}Y_{n,i},\  U_{n,k}^{(2)}= U_{n,k}- U_{n,k}^{(1)}.\]
	By \eqref{th28} in Lemma \ref{LQR01}, \eqref{pr01} holds if we can show that
	\begin{align}\label{pr03}
	\sum_{n=1}^\infty \dfrac{1}{n}\mathbb{E}\left(\dfrac{\|S_n-U_{n,n}^{(1)}\|}{n^{1/p}}\right)^q<\infty.
	\end{align}
	Since $0<q<1$, we have
	\begin{align}\label{pr05}
	\begin{split}
	\mathbb{E}\left(\|S_n-U_{n,n}^{(1)}\|^q\right)
	&\le n \mathbb{E}\left(\|X\|^q\mathbf{1}(\|X\|^q>{u_{n}^q})\right)\\
	& = n u_{n}^q \mathbb{P}(\|X\|^q>{u_{n}^q}) +n\int_{u_{n}^q}^{\infty} \mathbb{P}(\|X\|^q>t)\mathrm{d} t\\
	& \le u_{n}^q+n\int_{u_{n}^q}^{\infty} \mathbb{P}(\|X\|^q>t)\mathrm{d} t.
	\end{split}
	\end{align}
	Noting that for $n\ge 1$, $u_{n}^q$ is the quantile of order $1-1/n$ of $\|X\|^q$. Letting $u_0=0$, it thus follows from \eqref{pr05} that
	\begin{align*}%\label{pr07}
	\begin{split}
	&\sum_{n=1}^\infty \dfrac{1}{n}\mathbb{E}\left(\dfrac{\|S_n-U_{n,n}^{(1)}\|}{n^{1/p}}\right)^q\\
	&\le \sum_{n=1}^\infty \dfrac{u_{n}^q}{n^{1+q/p}}+ \sum_{n=1}^\infty \dfrac{1}{n^{q/p}}\int_{u_{n}^q}^{\infty} \mathbb{P}(\|X\|^q>t)\mathrm{d} t\\
	&\le \sum_{n=1}^\infty \dfrac{1}{n^{1+q/p}}\sum_{k=1}^n (u_{k}^q-u_{k-1}^q)
	+\sum_{n=1}^\infty \dfrac{1}{n^{q/p}}\sum_{k=n}^\infty \int_{u_{k}^q}^{u_{k+1}^q} \mathbb{P}(\|X\|^q>t)\mathrm{d} t\\
	&= \sum_{k=1}^\infty(u_{k}^q-u_{k-1}^q)\left(\sum_{n=k}^\infty  \dfrac{1}{n^{1+q/p}}\right)
	+\sum_{k=1}^\infty \int_{u_{k}^q}^{u_{k+1}^q} \mathbb{P}(\|X\|^q>t)\mathrm{d} t\left(\sum_{n=1}^k\dfrac{1}{n^{q/p}}\right)\\
	&\le \left(1+\dfrac{p}{q}\right)\sum_{k=1}^\infty\dfrac{1}{k^{q/p}}(u_{k}^q-u_{k-1}^q)
	+\dfrac{p}{p-q}\sum_{k=1}^\infty k^{1-q/p} \int_{u_{k}^q}^{u_{k+1}^q} \mathbb{P}(\|X\|^q>t)\mathrm{d} t\\
	&\le \left(1+\dfrac{p}{q}\right)\sum_{k=1}^\infty\int_{u_{k-1}^q}^{u_{k}^q}\mathbb{P}^{q/p}(\|X\|^q>t)\mathrm{d} t
	+\dfrac{p}{p-q}\sum_{k=1}^\infty\int_{u_{k}^q}^{u_{k+1}^q}\mathbb{P}^{q/p}(\|X\|^q>t)\mathrm{d} t\\
	& \le \left(1+\dfrac{p}{q}	+\dfrac{p}{p-q}\right) \int_{0}^\infty \mathbb{P}^{q/p}(\|X\|^q>t)\mathrm{d} t<\infty
	\end{split}
	\end{align*}
	thereby proving \eqref{pr03}.
	
	Now we consider the case where $0<q=p<1$. Since $\mathbb{E}(\|X\|^p)<\infty$, it is easy to see that
	\begin{equation}\label{pr07}
	\lim_{n\to\infty}\dfrac{u_n}{n^{1/p}}=0,
	\end{equation}
	and
	\begin{equation}\label{pr09}
	\sum_{n=1}^\infty \mathbb{P}(\|X_n\|^p>n)=\sum_{n=1}^\infty \mathbb{P}(\|X\|^p>n)<\infty.
	\end{equation}
	From \eqref{pr07}, we can assume that $u_n<n^{1/p}$ for all $n\ge 1$.
	We then write
	$$S_n=U_{n,n}^{(1)}+U_{n,n}^{(2)}+\sum_{k=1}^n X_k\mathbf{1}(\|X_k\|^p>n),\ n\ge 1.$$
	By the Borel-Cantelli lemma, it follows from
	\eqref{pr09} that
	\begin{equation*}
	\mathbb{P}(\|X_n\|^p> n\text{ i.o.}(n))=0
	\end{equation*}
	and hence
	\begin{equation*}
	\mathbb{P}\left(\max_{1\le k\le n}\|X_k\|^p> n\text{ i.o.}(n)\right)=0.
	\end{equation*}
	We thus have
	\begin{equation}\label{pr11}
	\sum_{n=1}^\infty \dfrac{1}{n}\left(\dfrac{\left\|\sum_{k=1}^n X_k\mathbf{1}(\|X_k\|^p>n)\right\|}{n^{1/p}}\right)^p<\infty \text{ a.s.}
	\end{equation}
	By using \eqref{pr11}, \eqref{th16} (with $0<q=p<1$) holds if we can show that
	\begin{align}\label{pr13}
	\sum_{n=1}^\infty \dfrac{1}{n}\mathbb{E}\left(\dfrac{\|U_{n,n}^{(1)}\|}{n^{1/p}}\right)^p<\infty,
	\end{align}
	and
	\begin{align}\label{pr15}
	\sum_{n=1}^\infty \dfrac{1}{n}\mathbb{E}\left(\dfrac{\|U_{n,n}^{(2)}\|}{n^{1/p}}\right)^p<\infty.
	\end{align}
	By \eqref{th29} in Lemma \ref{LQR01}, we have
	\[\sum_{n=1}^\infty \dfrac{1}{n^2}\mathbb{E}(\|U_{n,n}^{(1)}\|^p)<\infty\]
	which proving \eqref{pr13}. 
	Since $0<p<1$, it follows from \eqref{m03} (with $p=q$) that
	\begin{align*}%\label{pr15}
	\begin{split}
	\sum_{n=1}^\infty \dfrac{1}{n^2}\mathbb{E}\left(\|U_{n,n}^{(2)}\|^p\right)
	&= \sum_{n=1}^\infty \dfrac{1}{n^2}\mathbb{E}\left(\left\|\sum_{k=1}^n X_k\mathbf{1}\left( u_{n}<\|X_k\|\le n^{1/p}\right)\right\|^p\right)\\
	&\le \sum_{n=1}^\infty \dfrac{1}{n} \mathbb{E}\left(\left\|X\right\|^p\mathbf{1}\left( u_{n}<\|X\|\le n^{1/p}\right)\right)<\infty
	\end{split}
	\end{align*}
	which proving \eqref{pr15}.
\end{proof}

\begin{proof}[Proof of Theorem \ref{thm14} (Necessity)] By Proposition \ref{prop12}, we only need to prove for the case
	where $0<q=p<1$. Also by applying Proposition \ref{prop12}, we have from \eqref{th16} that
\[\lim_{n\to\infty}\dfrac{S_n}{n^{1/p}}=0\ \text{ a.s.,}\]
	which ensures $\mathbb{E}(\|X\|^p)<\infty$. Therefore, we only need to show that \eqref{th16} implies
	\begin{equation}\label{pr16}
	\sum_{n=1}^\infty \dfrac{\mathbb{E}(\|X\|^p\mathbf{1}(u_n^p<\|X\|^p\leq n))}{n}<\infty.
	\end{equation}
	Since $0<p<1,$
	\begin{align}\label{pr17}
	\begin{split}
	\sum_{n=1}^\infty \dfrac{1}{n^2}\mathbb{E}\left(\left\|U_{n,n}-\sum_{k=1}^nX_k \mathbf{1}(\|X_k\|^p\le k)\right\|^p\right)
	&\le \sum_{n=1}^\infty \dfrac{1}{n^2} \sum_{k=1}^n \mathbb{E}\left(\left\|X_k\right\|^p \mathbf{1}(k<\|X_k\|^p\le n)\right)\\
	& = \sum_{n=1}^\infty \dfrac{1}{n^2} \sum_{k=1}^n \sum_{j=k+1}^n \mathbb{E}\left(\left\|X\right\|^p \mathbf{1}(j-1<\|X\|^p\le j)\right)\\
	& \le \sum_{n=1}^\infty \dfrac{1}{n^2} \sum_{j=1}^n j\mathbb{E}\left(\left\|X\right\|^p \mathbf{1}(j-1<\|X\|^p\le j)\right)\\
	&= \sum_{j=1}^\infty j\mathbb{E}\left(\left\|X\right\|^p \mathbf{1}(j-1<\|X\|^p\le j)\right)\left(\sum_{n=j}^\infty \dfrac{1}{n^2} \right)\\
	&\le 2\sum_{j=1}^\infty \mathbb{E}\left(\left\|X\right\|^p \mathbf{1}(j-1<\|X\|^p\le j)\right)\\
	&=2\mathbb{E}\left(\|X\|^p\right)<\infty.
	\end{split}
	\end{align}
	Using the same argument of proof of (3.3) in \cite[p. 556]{LQR16}, we have
	\begin{equation}\label{pr19}
	\sum_{n=1}^\infty \dfrac{\mathbb{E}\left(\|\sum_{k=1}^{n}X_k\mathbf{1}(\|X_k\|^p\le k)\|^p\right)}{n^2}<\infty.
	\end{equation}
	Combining \eqref{pr17} and \eqref{pr19}, we have
	\begin{equation}\label{pr21}
	\sum_{n=1}^\infty \dfrac{\mathbb{E}\left(\left\|U_{n,n}\right\|^p\right)}{n^2}<\infty.
	\end{equation}
	It follows from \eqref{pr21} and \eqref{th29} of Lemma \ref{LQR01}  that
	\begin{equation}\label{pr26}
\begin{split}	
\sum_{n=1}^\infty \dfrac{\mathbb{E}\left(\|U_{n,n}^{(2)}\|^p\right)}{n^2}&=\sum_{n=1}^\infty \dfrac{\mathbb{E}\left(\|U_{n,n}-U_{n,n}^{(1)}\|^p\right)}{n^2}\\
	&\le \sum_{n=1}^\infty \dfrac{\mathbb{E}\left(\|U_{n,n}\|^p\right)}{n^2}+\sum_{n=1}^\infty \dfrac{\mathbb{E}\left(\|U_{n,n}^{(1)}\|^p\right)}{n^2}<\infty.
\end{split}	
\end{equation}
	Let $\left\{X^{'}, X_n^{'}, n\ge 1\right\}$ be an independent copy of $\left\{X, X_n, n \ge 1\right\}.$ For $n \ge 1, 1 \le k \le n,$ set
	\[V_{n,k}=X_k\mathbf{1}(u_n^{p}<\|X_k\|^p\le n)-X_k^{'}\mathbf{1}(u_n^{p}<\|X_k^{'}\|^p\le n),\ \hat{U}_{n,k}^{(2)}=\sum_{j=1}^{k}V_{n,j},\ \hat{U}_{n,0}^{(2)}=0.\]
	It follows from \eqref{pr26} that
	\begin{equation}\label{pr27}
	\sum_{n=1}^\infty \dfrac{\mathbb{E}\left(\|\hat{U}_{nn}^{(2)}\|^p\right)}{n^2}<\infty.
	\end{equation}
	For $n \ge 1$, applying L\'{e}vy’s inequality (see, e.g., \cite[p. 47-48]{LedouxTalagrand}) for independent symmetric random variables
	$\left\{V_{n,k}, 1 \le k \le n\right\}$, we have
	\begin{equation}\label{pr23}
	\begin{split}
	\mathbb{E}\left(\max_{1\le k\le n}\|V_{n,k}\|^p\right)
	&= \mathbb{E}\left(\max_{1\le k\le n}\left\|\hat{U}_{n,k}^{(2)}-\hat{U}_{n,(k-1)}^{(2)}\right\|^p \right)\\
	&\le 2\mathbb{E}\left(\max_{1\le k\le n}\left\|\hat{U}_{n,k}^{(2)}\right\|^p \right)\\
	&\le 4\mathbb{E}\left(\left\|\hat{U}_{n,n}^{(2)}\right\|^p \right).\\
	\end{split}
	\end{equation}
	Since
	\[\mathbb{P}(\|V_{n,1}\|^p>0)\le \mathbb{P}(\|X_1\|>u_{n})+\mathbb{P}(\|X_1^{'}\|>u_{n})<\dfrac{2}{n},\]
	by applying Lemma \ref{lem24} with the constant $K = 2$, we obtain
	\begin{equation}\label{pr29}
	\mathbb{E}(\|V_{n,1}\|^p)\le \dfrac{4}{n}\mathbb{E}\left(\max_{1\le k\le n}\|V_{n,k}\|^p\right)
	\end{equation}
	Combining \eqref{pr27}--\eqref{pr29}, we have
	\begin{equation}\label{pr30}
	\sum_{n=1}^{\infty}\dfrac{\mathbb{E}(\|V_{n,1}\|^p)}{n}\le\sum_{n=1}^{\infty}\dfrac{16\mathbb{E}\left(\|\hat{U}_{n,n}^{(2)}\|^p\right)}{n^2}<\infty.
	\end{equation}
	We see that $V_{n1}$ is a symmetrized version of $X_1\mathbf{1}(u_n^{p} <\|X_1\|^p\le n), n \ge 1$. Applying Lemma \ref{lem25} with $t = 1/n$ and $g(x) = \|x\|^p, x\in\mathbf{B}$, we have
	\begin{equation}\label{prr31}
	\mathbb{P}\left(\|X_1\|^p\mathbf{1}(u_n^{p} <\|X_1\|^p\le n)\le\dfrac{1}{n}\right)\mathbb{E}(\|X_1\|^p\mathbf{1}(u_n^{p} <\|X_1\|^p\le n))\le\mathbb{E}(\|V_{n,1}\|^p)+\dfrac{1}{n}.
	\end{equation}
	Since
	\begin{align*}
	1-\dfrac{1}{n}&\le \mathbb{P}(\|X\|\le u_n)\le \mathbb{P}\left(\|X\|^p\mathbf{1}(u_n^{p} <\|X\|^p\le n)\le\dfrac{1}{n}\right)\\
	&=\mathbb{P}\left(\|X_1\|^p\mathbf{1}(u_n^{p} <\|X_1\|^p\le n)\le\dfrac{1}{n}\right)\ \text{ for all }\ n\ge 1,
	\end{align*}
	it follows from \eqref{prr31} that
	\begin{equation}\label{phr32}
	\left(1-\dfrac{1}{n}\right)\mathbb{E}(\|X\|^p\mathbf{1}(u_n^{p}<\|X\|^p\le n))\le\mathbb{E}(\|V_{n1}\|^p)+\dfrac{1}{n}, n\ge 1.
	\end{equation}
	Combining \eqref{pr30} and \eqref{phr32}, we have
	\begin{equation*}
	\begin{split}
	\sum_{n=2}^{\infty}\dfrac{1}{2n}\mathbb{E}(\|X\|^p\mathbf{1}(u_n^{p}<\|X\|^p\le n))&\le\sum_{n=2}^{\infty}\dfrac{1}{n}\left(1-\dfrac{1}{n}\right)\mathbb{E}(\|X\|^p\mathbf{1}(u_n^{p}<\|X\|^p\le n))\\
	&\le\sum_{n=2}^{\infty}\dfrac{1}{n}\left(\mathbb{E}\left(\|V_{n,1}\|^{p}\right)+\dfrac{1}{n}\right)<\infty
	\end{split}
	\end{equation*}
	thereby completing the proof of \eqref{pr16}.
\end{proof}
\begin{proof}[Proof of Theorem \ref{thm15}] Since $0<q<1\le p<2$, the necessity follows immediately from Proposition \ref{prop12}
	and the fact that \eqref{th14} implies $\E(X)=0$. 
	The sufficiency follows from the implication (i)$\Rightarrow$(ii) of Theorem \ref{thm:stable11} in Section \ref{sec:stable}.
\end{proof}

\begin{proof}[Proof of Corollary \ref{cor01}]
Recalling Proposition \ref{prop12}, if $0<q<p<2$, then \eqref{m09} is equivalent to $X\in\SLLN(p,q)$.
	Therefore, the case where $0<q<p<1$ follows from Theorem \ref{thm14}, and the case where $0<q<1\le p<2$ follows from Theorem \ref{thm15}.
	
	We now consider the case where $0<q=p<1$. If \eqref{m09} holds, then by applying Proposition \ref{prop12} again, we obtain \eqref{m11} (with $q=p$).
	Conversely, if \eqref{m11} (with $q=p$) holds, then by following the proof of Lemma 5.6 of Li, Qi, and Rosalsky \cite{LQR11} with $\|X\|^p$ in the
	place of $\|X\|$, we obtain 
	\[\sum_{n=1}^\infty \dfrac{\mathbb{E}\left(\|X\|^p \mathbf{1}(\min\{u_{n}^p,n\}< \|X\|^p\le n)\right)}{n}<\infty.\]
	Therefore, \eqref{m03} (with $q=p$) holds, and by applying Theorem \ref{thm14}, we have $X\in\SLLN(p,p)$, i.e., \eqref{th16} holds with $q=p$.
	The conclusion \eqref{m09} then follows from Proposition \ref{prop12}.
\end{proof} 

We close this section by presenting three simple examples to illustrate Theorem \ref{thm14}, as mentioned in Remark \ref{rem01}. 
The first example shows that, for $0<p<1$, there exists a random variable $X$ such that $X\in\SLLN(p,p)$ but $X\notin\SLLN(p,q)$ for all $0<q<p$.

\begin{example}\label{t01} 
		Let $0<p<1$. For $q>0$, let $X$ be a real-valued random variable such that
		its tail probability function is
		\[\P\left(X>t\right)={\mathbf{1}}\{t\le\mathrm{e}\}+\dfrac{{\mathrm{e}}^q}{t^q(\ln t)^{2p/q}}{\mathbf{1}}\{t>\mathrm{e}\},\ t\in\R .\]
		Then for all $t>\mathrm{e}^q$, we have
		\[\P\left(|X|^q>t\right)=\P(X>t^{1/q})=\dfrac{{\mathrm{e}}^q}{t(\ln t^{1/q})^{2p/q}}.\]
		Therefore
		\begin{equation*}
		\int_{0}^{\infty}\P^{q/p}(|X|^q>t)\mathrm{d} t
		=\infty\ \text{ if }\ q<p.
		\end{equation*}
		For $q=p$, elementary calculus also shows that
		\begin{equation}\label{r07}
		\E\left(|X|^p \ln^{1/2} (1+|X|^p)\right)<\infty.
		\end{equation}
		The proof of Lemma 5.6 of Li, Qi, and Rosalsky \cite{LQR11} shows that, for any random variable $X$, if 
		\[\E\left(\|X\| \ln^{\delta} (1+\|X\|)\right)<\infty\ \text{ for some }\ \delta>0,\]
		then 
		\[\sum_{n=1}^\infty \dfrac{\mathbb{E}\left(\|X\| \mathbf{1}(\min\{u_{n},n\}< \|X\|\le n)\right)}{n}<\infty.\]
		It thus follows from \eqref{r07} that \eqref{m03} holds for $q=p$.
		By Theorem \ref{thm14} we see that, for this example,  $X\in\SLLN(p,p)$ but $X\notin\SLLN(p,q)$ for all $0<q<p$.
\end{example}

The next two examples show that each of the two conditions that
appeared in \eqref{m03} (for the case where $p=q$) does not imply each other. 
Examples \ref{t03} and \ref{t05} are inspired by Examples 5.2 and 5.3 of Li, Qi, and Rosalsky \cite{LQR11}, respectively.

\begin{example}\label{t03}

		Let $0<p<1$ and let $X$ be a real-valued random variable such that its tail probability function is
		\[\P\left(X>t\right)={\mathbf{1}}\{t\le {\mathrm{e}}^{\mathrm{e}}\}+\dfrac{{\mathrm{e}}^{\mathrm{e}p+1}}{t^p(\ln t)(\ln \ln t)^2}{\mathbf{1}}\{t>{\mathrm{e}}^{\mathrm{e}}\},\ t\in\R.\]
		Then $\E(|X|^p)<\infty$ and by the same calculation as in Lemma 5.2 of Li, Qi, and Rosalsky \cite{LQR11}, we have
		\[\sum_{n=1}^\infty \dfrac{\mathbb{E}\left(|X|^p \mathbf{1}(\min\{u_{n}^p,n\}<|X|^p\le n)\right)}{n}=\infty.\]
		
\end{example}

\begin{example}\label{t05}
		Let $0<p<1$ and let $X$ be a real-valued random variable such that its tail probability function is
		\[\P\left(X>t\right)={\mathbf{1}}\{t\le 1\}+\dfrac{1}{t^p}{\mathbf{1}}\{t>1\},\ t\in\R.\]
		Then $\E(|X|^p)=\infty$ and 
		\[\sum_{n=1}^\infty \dfrac{\mathbb{E}\left(|X|^p \mathbf{1}(\min\{u_{n}^p,n\}< |X|^p\le n)\right)}{n}=0\]
		since $u_{n}^p=n$.  
\end{example}

\section{Further remarks}\label{sec:remark}
This work has been devoted to $(p,q)$-type SLLN and related results for one-parameter processes.
	As noted by Khoshnevisan \cite{Khoshnevisan}, 
	``there are a number of compelling reasons for studying random fields,
	one of which is that, if and when possible, multiparameter processes are a
	natural extension of existing one-parameter processes''.
	Some of the tools used in this paper such as the generalization
	of Ottaviani’s inequality developed by Li and Rosalsky \cite{LiRosalsky13}
	or Lemma \ref{RTV} are available for
	multiparameter processes (see \cite{ATT,ThanhThuy}), but it is unclear whether
	the methods of this paper can be pushed through.


\begin{thebibliography}{}
	
	\bibitem{ATT}
	V. T. N. Anh, L. V. Thanh, and N. T. Thuy, On generalizations of maximal Inequalities for double arrays of independent random elements in Banach spaces,
	Preprint, Availaible at \url{ftp://file.viasm.org/Web/TienAnPham-16/Preprint_1640.pdf}
	
%	\bibitem{Chow88}
%	Y. S. Chow, On the rate of moment convergence of sample sums and extremes,
%	\textit{ Bull. Inst. Math. Acad. Sinica} \textbf{ 16} (1988), no. 3, 177--201.
	
	\bibitem{Acosta81}
	A. de Acosta, Inequalities for $B$-valued random vectors with
	applications to the strong law of large numbers, \textit{Ann. Probab.}
	\textbf{9} (1981), 157--161.
	
	\bibitem{EinmahlLi08}
	U. Einmahl and D.L. Li, Characterization of LIL behavior in Banach space, \textit{Trans. Amer. Math. Soc.} \textbf{360} (2008), no. 12, 6677--6693.
	
	\bibitem{Etemadi85}
	N.  Etemadi,  Tail probabilities for sums of independent Banach space
	valued random variables, \textit{Sankhy\={a} Ser. A} \textbf{47} (1985), 209–214.
	
	
	\bibitem{GineZinn83}
	E. Gin\'{e} and J. Zinn, 
	Central limit theorems and weak laws of large numbers in certain Banach spaces,
	\textit{Z. Wahrsch. Verw. Gebiete} \textbf{62} (1983), no. 3, 323--354.
	
	\bibitem{Hechner}
	F. Hechner, \textit{ Lois Fortes des Grands Nombres et Martingales Asymptotiques}, Doctoral thesis,
	l’Université de Strasbourg, France (2009).
	
	\bibitem{HechnerHeinkel}
	F. Hechner and B. Heinkel, 
	The Marcinkiewicz--Zygmund LLN in Banach spaces: a generalized martingale approach, 
	\textit{ J. Theoret. Probab.} \textbf{ 23} (2010), no. 2, 509--522.
	
	\bibitem{HoffmannPisier}
	J. Hoffmann-J\o rgensen and G. Pisier, The law of large numbers and the central limit theorem
	in Banach spaces, \textit{Ann. Probab.} \textbf{4} (1976), no. 4, 587--599.
	
	\bibitem{Khoshnevisan}
	D. Khoshnevisan, \textit{Multiparameter processes. An introduction to random fields}, 
	Springer Monographs in Mathematics. Springer--Verlag, New York, 2002. xx+584 pp.
	
\bibitem{KuelbsZinn79}
	J. Kuelbs and J. Zinn, 
Some stability results for vector valued random variables, \textit{Ann. Probab.} \textbf{7} (1979), no. 1, 75--84.

\bibitem{LedouxTalagrand88}
M. Ledoux and M. Talagrand Characterization of the law of the iterated logarithm in Banach
space, \textit{Ann. Probab.} \textbf{16}, (1988), 1242--1264.

\bibitem{LedouxTalagrand}
M. Ledoux, M. Talagrand,  Probability in Banach Spaces: Isoperimetry and Processes,
Springer, Berlin, 1991.

\bibitem{LiRosalsky13}{}%
	D. Li  and A. Rosalsky,   New versions of some classical stochastic inequalities, \textit{ Stoch.
		Anal. Appl.} {\bf31} (2013), no. 1, 62 - 79.
	
	\bibitem{LQR11}{}%
	D. Li,  Y. Qi,  and A. Rosalsky,   
	A refinement of the Kolmogorov-Marcinkiewicz--Zygmund strong law of large numbers, \textit{ J. Theoret. Probab.}
	\textbf{ 24} (2011), no. 4, 1130--1156.
	
	\bibitem{LQR15}{}%
	D. Li,  Y. Qi,  and A. Rosalsky,   An extension of theorems of Hechner and Heinkel. Asymptotic
	Laws and Methods in Stochastics: A Volume in Honour of Mikl{\'o}s Cs\"{o}rgo, Fields Institute
	Communications Series, Springer-Verlag, New York, 2015.
	
	
	\bibitem{LQR16}{}%
	D. Li,  Y. Qi,  and A. Rosalsky,  A characterization of a new type of strong law of large numbers,
	\textit{ Trans.
		Amer. Math. Soc.} \textbf{ 368} (2016), no. 1, 539 -- 561.
	
	
	\bibitem{MarcusPisier}
	M. Marcus and G. Pisier, Characterizations of almost surely continuous $p$-stable random
	Fourier series and strongly stationary processes, \textit{Acta Math.} \textbf{152} (1984), no. 3-4, 245--301.
	
	\bibitem{MarcusWoyczynski}
	M. Marcus and W. Woyczy\'{n}ski, Stable measures and central limit theorems in spaces of stable type, 
	\textit{Trans. Amer. Math. Soc.} \textbf{251} (1979), 71--102.

\bibitem{MikoschNorvaisa}	
T. Mikosch and R. Norvai\v{s}a, Strong laws of large numbers for fields of Banach space valued random variables, 
\textit{Probab. Theory Related Fields}, \textbf{74} (1987), no. 2, 241--253.

	\bibitem{Pisier}
	G. Pisier, Probabilistic methods in the geometry of Banach spaces, 
	\textit{Probability and analysis (Varenna, 1985)}, 167--241, 
	Lecture Notes in Math., 1206, Springer, Berlin, 1986.
	
	\bibitem{Pisier16}
	G. Pisier, \textit{Martingales in Banach spaces}, 
	Cambridge Studies in Advanced Mathematics, 155. Cambridge University Press, Cambridge, 2016. xxviii+561 pp.
	
	
	\bibitem{RTV}{}%
	A. Rosalsky, L. V. Thanh, and A. Volodin, On complete convergence in
	mean of normed sums of independent random elements in
	Banach spaces, \textit{ Stoch. Anal. Appl.} {\bf24} (2006), no. 1, 23--35.
	
	
	\bibitem{ThanhThuy}
	L. V. Thanh and N. T. Thuy, On complete convergence in mean for double sums of independent random elements in Banach spaces,
	\textit{ Acta Math. Hungar.} \textbf{ 150} (2016), no. 2, 456--471.
	
	\bibitem{Woyczynski}
	W. Woyczy\'{n}ski,
	Geometry and martingales in Banach spaces. II. Independent increments. \textit{Probability on Banach spaces}, pp. 267--517, 
	Adv. Probab. Related Topics, 4, \textit{Dekker, New York}, 1978.
	
	
\end{thebibliography}
\end{document}